\documentclass{amsart}
\usepackage{amssymb}
\usepackage{amsmath}
\usepackage{hyperref}
\usepackage{amsmath, amsthm, amssymb, amscd, amsxtra,graphicx}
\usepackage{latexsym, amsfonts}
\usepackage{indentfirst}
\newtheorem{theorem}{Theorem}[section]
\newtheorem{lemma}{Lemma}[section]
 
 \newtheorem{corollary}{Corollary}[section]
 
 \newtheorem{example}{Example}[section]

 \newtheorem{proposition}{Proposition}[section]

\newcommand{\cod}{{\operatorname{cod}}}

\begin{document}

\title{Two results on character codegrees}
\date{}

\author{Yang Liu}

\author{Yong Yang}

\address{School of Mathematical Science, Tianjin Normal University, Tianjin 300387, P. R. China}
\email{liuyang@math.pku.edu.cn}

\address{Department of Mathematics, Texas State University, San Marcos, TX 78666, USA} \email{yang@txstate.edu}

\makeatletter

\makeatother





\maketitle

\textbf{Abstract.} \, {Let $G$ be a finite group and $\mathrm{Irr}(G)$ be the set of irreducible characters of $G$. The codegree of an irreducible character $\chi$ of the group $G$ is defined as  $\mathrm{cod}(\chi)=|G:\mathrm{ker}(\chi)|/\chi(1)$. In this paper, we study two topics related to the character codegrees. Let  $\sigma^c(G)$ be the maximal integer $m$ such that there is a member in $\mathrm{cod}(G)$ having $m$ distinct prime divisors, where $\mathrm{cod}(G)=\{\mathrm{cod}(\chi)|\chi\in \mathrm{Irr}(G)\}$. One is related to the codegree version of the  Huppert's $\rho$-$\sigma$ conjecture and we obtain the best possible bound for $|\pi(G)|$ under the condition $\sigma^c(G) = 2,3,$ and $4$ respectively.
 The other is related to the prime graph of character codegrees and we show that the codegree prime graphs of several classes of groups can be characterized only by graph theoretical terms.}



\Large
\section{Introduction}
 Throughout this paper, $G$ is  a finite group and $\mathrm{Irr}(G)$ is the set of irreducible characters of $G$.
   The concept of character codegree  (defined as $|G|/\chi(1)$ for any nonlinear irreducible character $\chi$ of $G$) was first introduced in \cite{ch} to characterize the structure of finite groups. However for nonlinear character $\chi\in \mathrm{Irr}(G/N)$, where $N$ is a nontrivial normal subgroup of $G$, $\chi$ will have two different codegrees when it is considered as the character of $G$ and $G/N$ respectively. So the concept was improved as $\mathrm{cod}(\chi)=|G:\mathrm{ker}(\chi)|/\chi(1)$ for any character $\chi\in \mathrm{Irr}(G)$ by Qian, Wang, and Wei in \cite{qww}.   Now many properties of codegree have been studied, such as the relationship between the codegrees and the element orders \cite{i,q}, codegrees of $p$-groups \cite{CroomeLewis,dl}, and groups with few codegrees ~\cite{abg,LiuYang01,LiuYang02}.

  In this paper, we study two topics related to the character codegrees. Set $\cod(G) =\{\mathrm{cod}(\chi) \ |\ \chi \in \mathrm{Irr}(G)\}$. Denote by $\pi(n)$ be the set of prime divisors of rational integer $n$ and we use $\pi(G)$ instead of $\pi(|G|)$. Let  $\sigma^c(G)$ be the maximal integer $m$ such that there is a member in $\rm{cod}$$(G)$ having $m$ distinct prime divisors. The first topic is about the following question which is raised in \cite[Question A]{qww} and  called the codegree version of the  Hupperts's $\rho$-$\sigma$ conjecture: Is there any constant $k$ (independent of $G$) such that $|\pi(G)| \leq k \sigma^c(G)$? In
particular, is it true that $|\pi(G)| \leq 4$ when $\sigma^c(G) = 2$?

We note that the first question was answered in \cite{yq} and further improved in \cite{Yang}. In this paper, we show that the second  question is not always true in general and we obtain the best possible bound for $|\pi(G)|$ under the condition $\sigma^c(G) = 2,3,$ and $4$ respectively.

 \begin{theorem} Let $G$ be a finite group, then

(i)   $|\pi(G)| \leq 5$ when $\sigma^c(G) = 2$.

(ii)  $|\pi(G)| \leq 8$ when $\sigma^c(G) = 3$.

(iii) $|\pi(G)| \leq 12$ when $\sigma^c(G) = 4$.
\end{theorem}

  The element order prime graph (Gruenberg-Kegel graph) of a finite group
is the (simple undirected) graph  whose vertices are the prime numbers dividing the order of the group, with two vertices being linked by an edge if and only if their product divides the order of some element of the group. Gruber et al. characterized the element order prime graph of solvable groups only using  graph theoretical terms in \cite{gk}. Recently  Florez  et al. gave complete characterizations for the element order prime graphs of several classes of groups, such as groups
of square-free order, meta-nilpotent groups, groups of cube-free order, and, for any $n \in N$,
solvable groups of $n^{th}$-power-free order in \cite{fh}.

The codegree prime graph built on $\cod(G)$, that we denote by $\Delta_{\cod}(G)$, is the (simple undirected) graph whose vertices are the prime divisors of the numbers in $\cod(G)$, and two distinct vertices $p$, $q$ are adjacent if and only if $pq$ divides some number in $\cod(G)$.
 The second topic of this paper concerns the relation of solvable groups with the corresponding prime graph of character codegrees and we obtain several  results parallel to the results in \cite{fh,gk}.

\section{The codegree analogue of the Huppert's $\rho$-$\sigma$ conjecture}

We first introduce some notation. Let $G$ be a group and $N\unlhd G$. For $\chi\in \mathrm{Irr}(G)$, $\mathrm{Irr}(\chi|_N)$ is the set of irreducible constitutes of the restriction of $\chi$ to $N$.

\begin{lemma}[Lemma 1.8 of \cite{k}]
Let $|G|=q_1\cdots q_nr_1\cdots r_m$ with mutually distinct primes $q_i,r_j$ $(n,m\in \mathbb{N})$. Assume that $|F(G)|=r_1r_2\cdots r_m$ and $p$ is a prime with $|G| \mid p-1$. Then there exists a faithful irreducible $GF(p)G$-module $V$ such that $F(G)$ acts fixed point freely on $V$.
\end{lemma}

\begin{example}[Example 1.9 of \cite{k}]
\end{example}

Let  $q_1, q_2, q_3, q_4, r_1, r_2, r_3, r_4$ be mutually distinct prime numbers where the $q_i$ may be arbitrarily chosen and the $r_i$ may have the following properties:

\centerline{$q_1 \mid r_1-1$,  $q_2 \mid r_2-1$,  $q_1q_2q_3 \mid r_3-1$,  $q_2q_4 \mid r_4-1$.}

The $r_i (i=1,2,3,4)$  surely exist because of the Dirichlet's theorem on primes in arithmetic progression. We want $G/G''$ to have the following structure:

$G/G'\cong Q_1\times Q_2\times Q_3\times Q_4$ with $Q_i$ cyclic of order $q_i$,

$G'/G''\cong R_1\times R_2\times R_3\times R_4$ with $R_i$ cyclic of order $r_i$, and $Q_i\leq N_G(R_j)$ for all $i, j$. Then  $Q_i$ can only act trivially or fixed point freely on $R_j$ $(i,j\in\{1,2,3,4\})$, and we demand:

$Q_1$ acts fixed point freely only on $R_1$ and $R_3$,

$Q_2$ acts fixed point freely only on $R_2$, $R_3$, and $R_4$,

$Q_3$ acts fixed point freely only on $R_3$, and finally

$Q_4$ acts fixed point freely only on $R_4$.

This is possible because of our number theoretical restriction on the $r_i$. Now let $s_1, s_2, s_3, s_4$ be mutually distinct primes such that for $i=1,2,3,4$, $s_i-1$ is divisible by $q_1q_2q_3q_4r_1r_2r_3r_4$. (By Dirichlet's theorem such $s_i$ exist.)  Then let $G''=V_1\times V_2\times V_3\times V_4$ with irreducible $GF(s_i)$-modules $V_i$, where the $V_i$ may be constructed as follows:
\begin{enumerate}
\item Let $V_1$ be the faithful irreducible $GF(s_1)Q_1R_1$-module (existing by Lemma 2.1), on which $R_1$ acts fixed point freely. Let $Q_2$, $Q_3$, and $Q_4$ act fixed point freely on $V_1$ by multiplication with elements of $GF(s_1)$ (i.e. $Q_2Q_3Q_4$ is represented by multiples of the identity matrix), and $R_2R_3R_4$ may act trivially on $V_1$. So we have

$C_{G/G''}(V_1)\cong R_2\times R_3\times R_4$, and

$G/C_G(V_1)\cong Q_1R_1\times Q_2\times Q_3\times Q_4$.

\item As in (1) let $V_2$ be the $GF(s_2)Q_2R_2$-module, on which $R_2$ acts fixed point freely, and we demand

$C_{G/G''}(V_2)\cong R_1\times R_3\times R_4$, and

$G/C_G(V_2)\cong Q_2R_2\times Q_1\times Q_3\times Q_4$, where $Q_1, Q_3$ and $Q_4$ act fixed point freely on $V_2$ by multiplication with elements of $GF(s_2)$.

\item Let $V_3$ be the  faithful, irreducible $GF(s_3)Q_1Q_2Q_3R_1R_2R_3$-module, on which $R_1R_2R_3=F(Q_1Q_2Q_3R_1R_2R_3)$ acts fixed point freely. Let $Q_4$ act fixed point freely on $V_3$ by multiplication with elements of $GF(s_3)$, and let $R_4$  act trivially on $V_3$. Then

$C_{G/G''}(V_3)\cong R_4$, and

$G/C_G(V_3)\cong Q_1Q_2Q_3R_1R_2R_3\times Q_4$.

\item As in (3) let $V_4$ be the $GF(s_4)Q_1Q_2Q_4R_1R_2R_4$-module, so that $R_1R_2R_4$ acts fixed point freely on it, and we demand

$C_{G/G''}(V_4)\cong R_3$, and

$G/C_G(V_3)\cong Q_1Q_2Q_4R_1R_2R_4\times Q_3$, where $Q_3$ acts fixed point freely on $V_4$ by multiplication with elements of $GF(s_4)$.
\end{enumerate}

\begin{proposition}
Let $G$ be the group constructed in Example 2.1. Then $|\sigma^c(G)|=4$.
\end{proposition}

\begin{proof}

It can be checked that $F(G)=V_1V_2V_3V_4$ and $F_2(G)/F(G)=F(G/F(G))=R_1R_2R_3R_4F(G)/F(G)$. Next we consider prime divisors of $\mathrm{cod}(\chi)$ case by case for any $\chi\in \mathrm{Irr}(G)$.\\

Case (a). $G''\not\leq \mathrm{ker}(\chi)$.

Choose $\lambda\in \mathrm{Irr}(\chi|_{F(G)})$. We may assume $\lambda=\lambda_1\cdot\lambda_2\cdot\lambda_3\cdot\lambda_4$, where $\lambda_i\in \mathrm{Irr}(V_i)$.\\

Subcase (a1). Suppose that $\lambda_i\neq 1$, $i=1,2,3,4$.

It can be checked that $I_G(\lambda)=F(G)$ and $\chi(1)=|G/F(G)|$. Thus $\pi(\mathrm{cod}(\chi))=\{s_1, s_2, s_3, s_4\}$.\\

Subcase (a2). Suppose that one of the four characters $\lambda_i$ is trivial and the other three characters are nontrivial.

If $\lambda_4=1$, it can be checked that  $I_G(\lambda)=F(G)R_4$ and $\chi(1)=|G/(F(G)R_4)|$. Since $V_4\leq \mathrm{ker} (\lambda)$,  we have  $V_4\leq \mathrm{ker} (\chi)$ and $\pi(\mathrm{cod}(\chi))\subseteq\{s_1, s_2, s_3, r_4\}$. Similarly, we can check for subcases of $\lambda_i=1$, $i=1,2,3$ and also obtain that $|\pi(\mathrm{cod}(\chi))|\leq 4$.\\

Subcase (a3). Suppose that two of the four characters $\lambda_i$ are trivial and the other two characters are nontrivial.

If $\lambda_3$ and $\lambda_4$ are trivial characters, it can be checked that $I_G(\lambda)\leq F(G)(R_3R_4)$ and $|G|/\chi(1)\mid|F(G)R_3R_4|$. Since $V_3V_4\leq \mathrm{ker} (\lambda)$, we have $V_3V_4\leq \mathrm{ker} (\chi)$ and $\pi(\mathrm{cod}(\chi))\subseteq\{s_1, s_2, r_3, r_4\}$. Similarly, we can check for other subcases and also obtain that $|\pi(\mathrm{cod}(\chi))|\leq 4$.\\

Subcase (a4). Suppose that three of the four characters $\lambda_i$ are trivial and the remained character is nontrivial.

If $\lambda_4\neq 1$, then $T=I_G(\lambda)\leq F(G)R_3(Q_1Q_2Q_4)$ and $\lambda$ extends to $\theta\in \mathrm{Irr}(T)$.  By Clifford's theorem, there exists $\beta\in \mathrm{Irr}(T/F(G))$ such that $\chi=(\theta\cdot\beta)^G$. Since $V_1V_2V_3\leq \mathrm{ker} (\chi)$, $$\mathrm{cod}(\chi)=\frac{|G|}{\mathrm{ker}(\chi)\chi(1)}=\frac{|G|}{|\mathrm{ker}(\chi)||G:T|\beta(1)}=
\frac{|T|}{|\mathrm{ker}(\chi)|\beta(1)}\mid |V_4R_3(Q_1Q_2Q_4)|.$$ If $T\not\leq F(G)R_3(Q_1Q_2Q_4)$ or $\beta(1)>1$, then $|\pi(\mathrm{cod}(\chi))|\leq 4$. Next we assume $T= F(G)R_3(Q_1Q_2Q_4)$
and $\beta(1)=1$. Since $R_3=[R_3,Q_1]\leq T'\leq \mathrm{ker}(\theta\cdot\beta)$ and $V_1V_2V_3R_3$ is a normal subgroup of $G$, we have $V_1V_2V_3R_3\leq \mathrm{ker}(\theta\cdot\beta)$ and $V_1V_2V_3R_3\leq \mathrm{ker} (\chi)$. Therefore
$\pi(\mathrm{cod}(\chi))\subseteq\{s_4, q_1, q_2, q_4\}$.\\

Case (b).  $G''\leq \mathrm{ker}(\chi)$.

Choose $\lambda\in \mathrm{Irr}(\chi|_{F_2(G)})$. We may assume $\lambda=\lambda_1\cdot\lambda_2\cdot\lambda_3\cdot\lambda_4$, where $\lambda_i\in \mathrm{Irr}(R_i)$. All cases can be handled with the same method as before and we take the following case as an example. Suppose that  $\lambda_1, \lambda_2$ are nontrivial characters and $\lambda_3, \lambda_4$ are trivial, then $I_G(\lambda)=F_2(G)Q_3Q_4$ and $\chi(1)=q_1q_2$. Since $F(G)R_3R_4\leq \mathrm{ker}(\chi)$, we have $\pi(\mathrm{cod}(\chi))\subseteq\{r_1, r_2, q_3, q_4\}$.
\end{proof}

By the same method, we can obtain the following result.

\begin{proposition}
Let $H=Q_1Q_2R_1R_2V_1$ and $K=Q_2Q_3Q_4R_2R_4V_1V_3V_4$, where $Q_i, R_j, V_k$ are from Example 2.1, then $|\sigma^c(H)|=2$ and $|\sigma^c(K)|=3.$
\end{proposition}

By the previous examples and the main results of ~\cite{k,Keller3}, as well as the main result of ~\cite{i}, Theorem 1.1 can be easily checked. Note that all those upper bounds are the best possible.

\section{The codegree prime graphs of solvable groups}

The following three results are about the prime graphs (element order version) of solvable groups.

\begin{proposition}[Theorem 2 of \cite{gk}]\label{solvableprime}
An unlabeled graph $F$ is isomorphic to the prime graph of some solvable group if and only if its
complement $\overline{F}$ is 3-colorable and triangle-free.
\end{proposition}

\begin{proposition}[Theorem 3.1 of \cite{fh}]\label{metanilpotentprime}
Let $\mathcal{C}$ be any class of groups such that

\centerline{\{$G$: $G$ has square-free order\}$\subseteq \mathcal{C}\subseteq$ \{$G$: $G$ metanilpotent\}.}

\noindent Then $F$ is isomorphic to the prime graph of some $G\in \mathcal{C}$ if and only if $\overline{F}$ is bipartite.

\end{proposition}

\begin{proposition}[Theorem 4.3 of \cite{fh}]
 $F$ is isomorphic to the prime graph of a solvable group of $n^{th}$-power-free order if and only
if $\overline{F}$ satisfies the following conditions.

(1) Triangle-free.

(2) There exists a $3$-coloring of $\overline{F}$ by Red, Green, and Blue and a way to label each red vertex by a distinct prime number such that for any $\pi$ a subset of those primes whose product is at least $n$, we have that in the canonical orientation, no blue vertex is simultaneously  the end of directed $2$-paths starting from each of the red vertices in $\pi$.

\end{proposition}

In \cite{q}, Qian proved that if $G$ is a solvable group and $g\in G$, then there is an irreducible character $\chi$ of $G$ such that $p$ divides $\mathrm{cod}(\chi)$ for any prime divisor $p$ of the order of $g$. A while later, Isaacs proved that the result is indeed true for any finite group in \cite{i}. In view of this result, we see that the element order prime graph of a finite group $G$ is a subgraph of its codegree prime graph. Next we show the element order prime graph and codegree prime graph  coincide for various types of solvable groups.

\begin{proposition}\label{codtoelement}
Let $G$ be a solvable group such that every minimal normal subgroup is a Sylow subgroup and $|G/F(G)|$ is square-free. For any prime pairs $(p,q)$ if $pq\mid \mathrm{cod}(\chi)$ for some
$\chi\in \mathrm{Irr}(G)$, then there exists an element $g\in G$ such that $pq\mid|g|$.
\end{proposition}

\begin{proof} We proceed by the induction on the order of $G$. If $\mathrm{ker}(\chi)$ is nontrivial, then it can be checked that every minimal normal subgroup of $G/\mathrm{ker}(\chi)$ is a Sylow subgroup and $(G/\mathrm{ker}(\chi))/F(G/\mathrm{ker}(\chi))$ is square-free. Then there exists an element $g\mathrm{ker}(\chi)\in G/\mathrm{ker}(\chi)$ such that $pq\mid|g\mathrm{ker}(\chi)|$ by induction. Therefore $pq\mid|g|$. Next we assume that  $\mathrm{ker}(\chi)=1$.

Let $F(G)=P_1\times P_2\times\cdots\times P_t$, where $P_i$ is a Sylow $p_i$-subgroup. Choose $\lambda\in \mathrm{Irr}(\chi|_{F(G)})$ and set $T=I_G(\lambda)$. Then $\chi(1)=|G:T|\theta(1)$, where $\theta\in \mathrm{Irr}(T|\lambda)$ by Clifford theory and thus $\mathrm{cod}(\chi)=|T|/\theta(1)$. If $pq\mid|F(G)|$, then there exists an element of order $pq$. If $p\mid |F(G)|$ and $(q,|F(G)|)=1$, then the order of $\lambda$ is divisible by $p$ and $q\mid |T|$. Since the actions of $G/F(G)$ on $\mathrm{Irr}(F(G))$ and $F(G)$ are permutation isomorphic (by Theorem 13.24 of ~\cite{i06}), there exists $h\in F(G)$ such that $p\mid |h|$ and $q\mid|C_G(h)|$. Hence there exists an element $g\in G$ such that $pq\mid|g|$. If  $q\mid |F(G)|$ and $(p,|F(G)|)=1$, the proposition is true similarly.
Now we can assume $pq\mid |G/F(G)|$. Let $F_2(G)/F(G)=F(G/F(G))$. Since $G/F(G)$ is square-free, $F_2(G)/F(G)$ and $G/F_2(G)$ are cyclic groups. If $pq \mid |F_2(G)/F(G)|$ or $pq \mid |G/F_2(G)|$, the proposition is true. Without loss of generality we assume that $p\mid |F_2(G)/F(G)|$ and $q \mid |G/F_2(G)|$. Choose a Sylow $p$-subgroup $P\leq T$ and set $N=F(G)P$.
Let $\alpha\in \mathrm{Irr}(\chi|_N)$ such that $\lambda\in \mathrm{Irr}(\alpha|_{F(G)})$, then $\alpha$ is an extension of $\lambda$. Since $N'\leq \mathrm{ker}(\alpha)$ and $N\unlhd G$, $N'\leq \mathrm{ker}(\chi)$ and $N'=1$. Hence $P\leq C_G(F(G))=F(G)$, a contradiction with $(p,|F(G)|)=1$.
\end{proof}

Inspired by the previous results, we obtain similar results on the prime graph of codegrees.

\begin{proposition}\label{codsolvableprime}
An unlabeled graph $F$ is isomorphic to the codegree prime graph of some solvable group if and only if its
complement $\overline{F}$ is 3-colorable and triangle-free.
\end{proposition}

\begin{proof}
First, we assume $F$  is isomorphic to the codegree prime graph of a solvable group $G$. Let $F_1$ be the prime graph of $G$. Then $\overline{F_1}$ is 3-colorable and triangle-free by Proposition \ref{solvableprime}. Since $F_1$ is subgraph of $F$, $\overline{F}$ is a subgraph of $\overline{F_1}$. Hence $\overline{F}$ is 3-colorable and triangle-free.

Conversely, if $\overline{F}$ is 3-colorable and triangle-free, then  a solvable group $G$ can be constructed such that $F$ is isomorphic to the prime graph of $G$ (The details on  the structure of  $G$  can be found in Theorem 2.8 of \cite{gk}). It can be checked that every minimal normal subgroup of $G$ is a Sylow subgroup and $|G/F(G)|$ is square-free.
Then  $\Delta_{\cod}(G)$ is the same as its prime graph by Proposition \ref{codtoelement}.
Therefore $F$ is isomorphic to $\Delta_{\cod}(G)$.
\end{proof}

\begin{proposition}\label{codmetanilpotentprime}
Let $\mathcal{C}$ be any class of groups such that

\centerline{\{$G$: $G$ has square-free order\}$\subseteq \mathcal{C}\subseteq$ \{$G$: $G$ metanilpotent\}.}

\noindent Then $F$ is isomorphic to the codegree prime graph of some $G\in \mathcal{C}$ if and only if $\overline{F}$ is bipartite.

\end{proposition}

\begin{proof}
First, we assume $F$ is isomorphic to the codegree prime graph of a solvable group $G$. Let $F_1$ be the prime graph of $G$. Then $\overline{F_1}$ is bipartite by Proposition \ref{metanilpotentprime}. Since $F_1$ is subgraph of $F$, $\overline{F}$ is a subgraph of $\overline{F_1}$. Hence $\overline{F}$ is bipartite.

Conversely, if $\overline{F}$ is 3-colorable and triangle-free, then  a solvable group $G$ of square-free order can be constructed such that $F$ is isomorphic to the prime graph of $G$ (The details on  the structure of  $G$  can be found in Theorem 3.1 of \cite{fh}).
Since $G$ has square-free order,  $\Delta_{\cod}(G)$ is the same as its prime graph. Therefore $F$ is isomorphic to $\Delta_{\cod}(G)$.
\end{proof}

\begin{proposition}
 $F$ is isomorphic to the codegree prime graph of a solvable group of $n^{th}$-power-free order if and only
if $\overline{F}$ satisfies the following conditions.

(1) Triangle-free;

(2) There exists a $3$-coloring of $\overline{F}$ by Red, Green, and Blue and a way to label each red vertex by a distinct prime number such that for any $\pi$ a subset of those primes whose product is at least $n$, we have that in the canonical orientation, no blue vertex is simultaneously  the end of directed $2$-paths starting from each of the red vertices in $\pi$.

\end{proposition}

\begin{proof}
First, we assume $F$ is isomorphic to the codegree prime graph of  solvable group $G$ with $n^{th}$-power-free order. Let $F_1$ be the prime graph of $G$. Then $\overline{F_1}$ satisfies the conditions in
Proposition 3.3, so does $\overline{F}$ for $\overline{F}$ is a subgraph of $\overline{F_1}$.

Conversely, if $\overline{F}$ satisfies the conditions, then by Theorem 4.3 of \cite{fh} a solvable group $G$ can be constructed such that $F$ is isomorphic to the prime graph of $G$  which is still from  Theorem 2.8 of \cite{gk}.
Then $\Delta_{\cod}(G)$ is the same as its prime graph by Proposition \ref{codtoelement}.
Therefore $F$ is isomorphic to $\Delta_{\cod}(G)$.
\end{proof}

Similar with Corollaries 4.4-4.7 in \cite{fh}, we can obtain the following results.

\begin{corollary}
 $F$ is isomorphic to the codegree prime graph of a solvable group of $n^{th}$-power-free order if and only
if $\overline{F}$ satisfies the following conditions.

(1) Triangle-free;

(2) There exists a 3-coloring of $\overline{F}$ by Red, Green, and Blue such that we can label the red vertices by the first m primes such that

\ \ (a) The primes are less than $n$;

\ \ (b) In the canonical orientation, no blue vertex is simultaneously  the end of directed $2$-paths starting from each of the red vertices in $\pi$, where $\pi$ is any subset of the first $m$ primes whose product is at least $n$.
\end{corollary}

\begin{corollary}
 $F$ is isomorphic to the codegree prime graph of a solvable group of square-free order if and only
if $\overline{F}$ is $2$-colorable.
\end{corollary}

\begin{corollary}
 $F$ is isomorphic to the codegree prime graph of a solvable group of cube free order if and only
if $\overline{F}$ is triangle free and is $2$-colorable after removing one vertex.
\end{corollary}

\begin{corollary} The following are equivalent.

 (1) $F$ is isomorphic to the codegree prime graph of a solvable group of fourth-power-free order.

 (2) $F$ is isomorphic to the codegree prime graph of a solvable group of fifth-power-free order.

 (3) $\overline{F}$ is $3$-colorable, triangle free and satisfies one of the following conditions.

\ \ (a) $\overline{F}$ is $2$-colorable after removing one vertex;

\ \ (b) There exists a $3$-coloring of $\overline{F}$  by Red, Green, and Blue such that there are exactly $2$ red vertices. Furthermore, in the canonical orientation, no vertex is the end of a directed $2$-path starting from both of the red vertices.
\end{corollary}

\bigskip
\noindent \textbf{Acknowledgments}

 The project is supported by NSFC (Grant Nos. 11671063, 11701421, and 11871011), the Simons Foundation (No. 499532), and the Science \& Technology Development  Fund of Tianjin Education Commission for Higher Education (2020KJ010).

\end{document}